\documentclass[11pt]{amsart}
\usepackage{amsfonts,amssymb,amscd,amsmath,enumerate,verbatim,calc,graphicx}

\def\NZQ{\mathbb}               

\def\ZZ{{\NZQ Z}}
\def\RR{{\NZQ R}}

\def\frk{\mathfrak}               

\def\Phi{{\frk N}}
\def\eb{{\bold e}}

\def\opn#1#2{\def#1{\operatorname{#2}}} 
\opn\chara{char} \opn\length{\ell} \opn\pd{pd} \opn\rk{rk}
\opn\projdim{proj\,dim} \opn\injdim{inj\,dim} \opn\rank{rank}
\opn\depth{depth} \opn\grade{grade} \opn\height{height}
\opn\embdim{emb\,dim} \opn\codim{codim}

\opn\Tr{Tr} \opn\bigrank{big\,rank}
\opn\superheight{superheight}\opn\lcm{lcm}
\opn\trdeg{tr\,deg}
\opn\reg{reg} \opn\lreg{lreg} \opn\ini{in} \opn\lpd{lpd}
\opn\size{size}\opn{\mult}{mult}
\opn\aff{aff} \opn\con{conv} \opn\relint{relint} \opn\st{st}
\opn\lk{lk} \opn\cn{cn} \opn\inte{int} \opn\vol{vol}
\opn\link{link} \opn\star{star}
\opn\gr{gr}

\def\Fc{{\mathcal F}}
\def\Pc{{\mathcal P}}

\def\lcl{\left\lceil}
\def\rcl{\right\rceil}
\def\lfr{\left\lfloor}
\def\rfr{\right\rfloor}

\newtheorem{Theorem}{Theorem}[section]
\newtheorem{Lemma}[Theorem]{Lemma}

\newtheorem{Remark}[Theorem]{Remark}

\newtheorem{Example}[Theorem]{Example}

\newtheorem{Problem}[Theorem]{Problem}

\newtheorem{Question}[Theorem]{Question}
\numberwithin{equation}{section}
\begin{document}

\title{Unimodality on $\delta$-vectors of lattice polytopes 
and two related properties}

\author{Akihiro Higashitani}
\thanks{
\;\;\;\; 
The author is partially supported by a JSPS Fellowship for Young Scientists and by JSPS Grant-in-Aid for Young Scientists (B) $\sharp$26800015. 
}
\address{Akihiro Higashitani,
Department of Mathematics, Graduate School of Science, 
Kyoto University, Kitashirakawa-Oiwake cho, Sakyo-ku, Kyoto, 606-8502, Japan}
\email{ahigashi@math.kyoto-u.ac.jp}

\begin{abstract}
In this paper, we investigate two properties concerning the unimodality of the $\delta$-vectors 
of lattice polytopes, which are log-concavity and alternatingly increasingness. 
For lattice polytopes $\Pc$ of dimension $d$, 
we prove that the dilated lattice polytopes $n\Pc$ have strictly log-concave and strictly alternatingly increasing 
$\delta$-vectors if $n > \max\{s,d+1-s\}$, where $s$ is the degree of the $\delta$-polynomial of $\Pc$. 
The bound $\max\{s,d+1-s\}$ for $n$ is reasonable. 
We also provide several kinds of unimodal (or non-unimodal) $\delta$-vectors. 
Concretely, we give examples of lattice polytoeps whose $\delta$-vectors are not unimodal, 
unimodal but neither log-concave nor alternatingly increasing, 
alternatingly increasing but not log-concave, and log-concave but not alternatingly increasing, respectively. 
\end{abstract}

\maketitle

\section{Introduction}
The $\delta$-vectors of lattice polytopes 
are one of the most fascinating objects on enumerative combinatorics. 
In this paper, we focus on the unimodality question on $\delta$-vectors of lattice polytopes 
and investigate two related properties on the unimodality, 
called ``log-concave'' and ``alternatingly increasing''. 

Let $\Pc \subset \RR^N$ be a {\em lattice} polytope of dimension $d$, 
which is a convex polytope all of whose vertices are lattice points in the lattice $\ZZ^N$. 
Given a positive integer $m$, we define $$i(\Pc,m) = |m\Pc \cap \ZZ^N|,$$ 
where $m\Pc = \{ m\alpha : \alpha \in \Pc \}$ and $|\cdot|$ denotes the cardinality. 
The enumerative function $i(\Pc,m)$ is actually a polynomial in $m$ of degree $d$ 
with its constant term 1 (\cite{Ehrhart}). 
This polynomial $i(\Pc,m)$ is called the {\em Ehrhart polynomial} of $\Pc$. 
Moreover, $i(\Pc,m)$ satisfies {\em Ehrhart--Macdonald reciprocity} (see \cite[Theorem 4.1]{BeckRobins}): 
\begin{align}\label{reciprocity}
|m \Pc^\circ \cap \ZZ^N| =( - 1 )^d i(\Pc, - m) \text{ for each integer }m > 0, 
\end{align}
where $\Pc^\circ$ denotes the relative interior of $\Pc$. 
We refer the reader to \cite[Chapter 3]{BeckRobins} or \cite[Part II]{HibiRedBook} 
for the introduction to the theory of Ehrhart polynomials.

We define the sequence $\delta_0, \delta_1, \ldots, \delta_d$ of integers by the formula 
\begin{eqnarray*}
(1 - t)^{d + 1}\left(1+  \sum_{m=1}^{\infty} i(\Pc,m) t^m \right)
= \sum_{i=0}^d \delta_i t^i. 
\end{eqnarray*}
We call the integer sequence $$\delta(\Pc)= (\delta_0, \delta_1, \ldots, \delta_d)$$ 
the {\em $\delta$-vector} (also called {\em Ehrhart $\delta$-vector} or $h^*$-vector) of $\Pc$ and 
the polynomial $\delta_\Pc(t)=\delta_0+\delta_1t+\cdots+\delta_dt^d$ the {\em $\delta$-polynomial} of $\Pc$.

By \eqref{reciprocity}, one has 
\begin{eqnarray*}
(1 - t)^{d + 1}\sum_{m=1}^\infty |m \Pc^\circ \cap \ZZ^N| t^m = \sum_{i=0}^d \delta_{d-i} t^{i+1}. 
\end{eqnarray*}
Thus it follows that 
\begin{align}\label{naibu}
\min\{ k : k\Pc^\circ \cap \ZZ^N \neq \emptyset \} = d+1- \max\{ i : \delta_i \neq 0 \}.
\end{align}

The $\delta$-vectors of lattice polytopes have the following properties: 
\begin{itemize}
\item $\delta_0=1$, $\delta_1 = |\Pc \cap \ZZ^N| - (d + 1)$ and 
$\delta_d = |\Pc^\circ \cap \ZZ^N|.$ 
Hence, $\delta_1 \geq \delta_d$. In particular, when $\delta_1=\delta_d$, 
$\Pc$ must be a simplex. 
\item Each $\delta_i$ is nonnegative (\cite{StanleyDRCP}). 
\item If $\Pc^\circ \cap \ZZ^N$ is nonempty, i.e., $\delta_d > 0$, 
then $\delta_1 \leq \delta_i$ for each $1 \leq i \leq d - 1$ (\cite{HibiLBT}). 
\item The leading coefficient $(\sum_{i=0}^d\delta_i)/d!$ of $i(\Pc,n)$ 
is equal to the volume of $\Pc$ (\cite[Corollary 3.20, 3.21]{BeckRobins}). 
\end{itemize}

\smallskip

There are two well-known inequalities on $\delta$-vectors. Let $s = \max\{ i : \delta_i \neq 0 \}$. One is 
\begin{eqnarray}\label{Stanley}
\delta_0 + \delta_1 + \cdots + \delta_i 
\leq \delta_s + \delta_{s-1} + \cdots + \delta_{s-i}, 
\;\;\;\;\; 0 \leq i \leq s, 
\end{eqnarray}
which is proved by Stanley \cite{StanleyJPAA}, and another one is 
\begin{eqnarray}\label{Hibi}
\delta_d + \delta_{d-1} + \cdots + \delta_{d-i} 
\leq \delta_1 + \delta_2 + \cdots + \delta_{i+1}, 
\;\;\;\;\; 0 \leq i \leq d-1, 
\end{eqnarray}
which appears in the work of Hibi \cite[Remark (1.4)]{HibiLBT}. 

\bigskip

For a lattice polytope $\Pc \subset \RR^N$, 
we say that $\Pc$ has the {\em integer decomposition property} (IDP, for short) if 
for each integer $\ell \geq 1$ and $\alpha \in \ell\Pc \cap \ZZ^N$, 
there are $\alpha_1,\ldots,\alpha_\ell$ in $\Pc \cap \ZZ^N$ such that $\alpha=\alpha_1+\cdots+\alpha_\ell$. 
Having IDP is also known as what is {\em integrally closed}.

\bigskip


We also recall the following three notions. Let $(a_0,a_1,\ldots,a_d)$ be a sequence of real numbers. 
\begin{itemize}
\item We say that $(a_0,a_1,\ldots,a_d)$ is {\em unimodal} if there is some $c$ with $0 \leq c \leq d$ such that 
$$a_0 \leq a_1 \leq \cdots \leq a_c \geq a_{c+1} \geq \cdots \geq a_d.$$ 
If each inequality is strict, then we say that it is {\em strictly unimodal}. 
\item $(a_0,a_1,\ldots,a_d)$ is called {\em log-concave} if for each $1 \leq i \leq d-1$, 
one has $$a_i^2 \geq a_{i-1}a_{i+1}.$$ 
If $a_i^2 > a_{i-1}a_{i+1}$ for each $i$, then it is called {\em strictly log-concave}. 
\item (\cite[Definition 2.9]{SchLan}) We call $(a_0,a_1,\ldots,a_d)$ {\em alternatingly increasing} if 
$a_i \leq a_{d-i}$ for $0 \leq i \leq \lfr (d-1)/2 \rfr$ and $a_{d+1-i} \leq a_i$ for $1 \leq i \leq \lfr d/2 \rfr$ are satisfied. 
Namely, 
$$a_0 \leq a_d \leq a_1 \leq a_{d-1} \leq \cdots \leq a_{\lfr (d-1)/2 \rfr} \leq a_{d- \lfr (d-1)/2 \rfr} \leq a_{\lfr (d+1)/2 \rfr}.$$ 
If each inequality is strict, then we call it {\em strictly alternatingly increasing}. 
\end{itemize}

Note that $(a_0,\ldots,a_d)$ is unimodal (resp. strictly unimodal) 
if it is log-concave (resp. log-concave) or alternatingly increasing (resp. strictly alternatingly increasing).


\bigskip

Our motivation to organize this paper is to give some answer for the following: 
\begin{Question}\label{motivation}
Let $\Pc$ be a lattice polytope having IDP with at least one interior lattice point. 
Then is $\delta(\Pc)$ always unimodal? 
\end{Question}
The similar question is also mentioned in \cite[Question 1.1]{SchLan}. 
Moreover, the following has been conjectured by Stanley \cite{StanleyLC} in 1989: 
the $h$-vectors of standard graded Cohen--Macaulay domains are always unimodal. 
This conjecture still seems to be open. 
We note that Question \ref{motivation} is the case of Ehrhart rings (see \cite[Part II]{HibiRedBook}) 
for this question with additional condition ``$a$-invariant $-1$''. 

For this question, 
the following facts on the unimodality of the $\delta$-vectors of lattice polytopes are known: 
\begin{enumerate}
\item If $\Pc \subset \RR^d$ is a {\em reflexive polytope} (introduced in \cite{Batyrev}), 
which is a lattice polytope whose dual polytope $\Pc^\vee=\{y \in \RR^d : \langle x,y\rangle \leq 1 \text{ for all }x \in \Pc\}$ 
is also a lattice polytope, of dimension at most 5, then $\delta(\Pc)$ is unimodal. 
This follows from \cite[Theorem 1.1]{HibiLBT} and \cite{HibiCombinatorica}. 
Note that every reflexive polytope contains exactly one interior lattice point. 
\item Hibi conjectured that all the $\delta$-vectors of reflexive polytopes are unimodal (\cite[\S 36]{HibiRedBook}). 
However, counterexamples were found by Musta\c{t}\u{a} and Payne \cite{MP, Payne}. 
On the other hand, their counterexamples do not have IDP. 
It may be still open whether there exista a reflexive polytope having IDP whose $\delta$-vector is not unimodal. 
\item Bruns and R\"omer \cite{BR} proved that each reflexive polytope with a regular unimodular triangulation 
has a unimodal $\delta$-vector. 
Note that if a lattice polytope has a regular unimodular triangulation, then it also has IDP, 
while the converse is not true in general. 
\item Schepers and Van Langenhoven \cite[Proposition 2.17]{SchLan} proved that every parallelepiped 
with at least one interior lattice point has an alternatingly increasing $\delta$-vector. 
\end{enumerate}

In this paper, as a further contribution for Question \ref{motivation}, we prove the following: 
\begin{Theorem}\label{meinseoremu}
Let $\Pc \subset \RR^N$ be a lattice polytope of dimension $d$ and $s$ the degree of $\delta_\Pc(t)$. 
Let $(\delta_0,\delta_1,\ldots,\delta_d)$ be the $\delta$-vector of the dilated polytope $n\Pc$ 
for a positive integer $n$. Then the following statements hold: 
\begin{itemize}
\item[(i)] $(\delta_0,\delta_1,\ldots,\delta_d)$ is strictly log-concave when $n \geq s$; 
\item[(ii)] We have $\delta_i \leq \delta_{d-i}$ for $1 \leq i \leq \lfr (d-1)/2 \rfr$ and 
$\delta_{d+1-i} < \delta_i$ for $1 \leq i \leq \lfr d/2 \rfr$ when $n \geq \max\{ s, d+1-s \}$. 
Moreover, if $|(d+1-s)\Pc^\circ \cap \ZZ^N| > 1$, 
then we have $\delta_i < \delta_{d-i}$ for $0 \leq i \leq \lfr (d-1)/2 \rfr$. 
Hence, $(\delta_0,\delta_1,\ldots,\delta_d)$ is strictly alternatingly increasing 
when $n \geq \max\{ s, d+1-s \}$ with $|(d+1-s)\Pc^\circ \cap \ZZ^N| > 1$, 
or when $n > \max\{ s, d+1-s \}$. 
\end{itemize}
\end{Theorem}

In \cite{BeckStapledon} and \cite{BrentiWelker}, it has been proved that 
for a lattice polytope $\Pc$ of dimension $d$, there exists an integer $n_d$ such that 
the $\delta$-vector of a dilated polytope $n\Pc$ are strictly log-concave and strictly alternatingly increasing 
for each $n \geq n_d$. Theorem \ref{meinseoremu} gives an explicit bound for $n_d$. 
Moreover, the following remark says that our bound $\max\{s,d+1-s\}$ is reasonable in some sense. 
\begin{Remark}\label{rima-ku}{\em 
By \eqref{naibu}, we see that $n\Pc^\circ \cap \ZZ^N \not= \emptyset$ if and only if $n \geq d+1-s$. 

Moreover, by \cite[Theorem 1.1]{CHHH}, the inequality $\mu_{\text{idp}}(\Pc) \leq \mu_{\text{Ehr}}(\Pc)$ holds. 
Since $\mu_{\text{Ehr}}(\Pc)=\max\{i : \delta_i \not= 0\}=s$, we see that $n\Pc$ has IDP if $n \geq s$. 
In addition, when $\mu_{\text{idp}}(\Pc) = \mu_{\text{Ehr}}(\Pc)$, this bound is sharp. 

Therefore, $n\Pc$ has IDP and contains at least one interior lattice point if $n \geq \max\{s,d+1-s\}$ 
and the bound $\max\{s,d+1-s\}$ sometimes becomes optimal. 
}\end{Remark}

For the proof of Theorem \ref{meinseoremu}, we prove a more general statement (Theorem \ref{main1}) in Section 2. 
(See Remark \ref{ippanka}, too.) 

Moreover, we also provide several kinds of $\delta$-vectors of lattice polytopes concerning the unimodality of $\delta$-vectors. 
We construct an infinite family of lattice polytopes whose $\delta$-vectors are not unimodal in Section 3.1, 
unimodal but neither log-concave nor alternatingly increasing for even dimensions in Section 3.2, 
alternatingly increasing but not log-concave in Section 3.3, and  
log-concave but not alternatingly increasing for law dimensions, respectively. (See Figure \ref{figure}.)

\begin{figure}[htb!]
\centering
\includegraphics[scale=0.6]{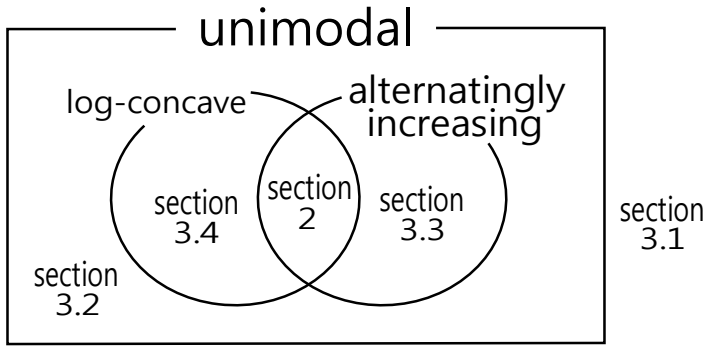}
\caption{Organization of this paper}\label{figure}
\end{figure}


\section{Unimodality on $\delta$-vectors of dilated polytopes}

%

We recall some notation from \cite{BeckStapledon}. 
For a polynomial $h(t)=h_0+h_1t+\cdots+h_dt^d$ in $t$ of degree at most $d$ 
with nonnegative coefficients and $h_0=1$, there is a polynomial $g(m)$ in $m$ of degree $d$ such that 
$$\sum_{m = 0}^\infty g(m)t^m = \frac{h(t)}{(1-t)^{d+1}}.$$ (See \cite[Lemma 3.9]{BeckRobins}.) 
Note that the polynomial $g(m)$ can be written like $g(m)=\sum_{i=0}^d h_i \binom{m+d-i}{d}$. 
For each integer $n$, we define $U_n h(t)=h_0^{(n)}+h_1^{(n)} t + \cdots + h_d^{(n)} t^d$ as follows: 
$$\sum_{m = 0}^\infty g(nm)t^m = \frac{U_n h(t)}{(1-t)^{d+1}}.$$

The main result of this paper is the following: 
\begin{Theorem}\label{main1}
Let $d \geq 5$ and let $h(t)=\sum_{j=0}^dh_jt^j$ be a polynomial in $t$ of degree $s$, where $s \leq d$, 
with nonnegative coefficients and $h_0=1$. 
Let $n$ be a positive integer and $\delta_i=h_i^{(n)}$ the coefficient of $U_nh(t)$ for $0 \leq i \leq d$. 
Then the following statements hold: 
\begin{itemize}
\item[(i)]
$(\delta_0,\delta_1,\ldots,\delta_d)$ is strictly log-concave when $n \geq s$; 
\item[(ii)] If $h_0,h_1,\ldots,h_d$ satisfy the inequalities 
\begin{align}\label{st}
h_0+h_1+\cdots+h_i \leq h_s+h_{s-1}+\cdots+h_{s-i} \;\text{ for }\; 0 \leq i \leq \lfr s/2 \rfr 
\end{align}
and 
\begin{align}\label{hi}
h_d+h_{d-1}+\cdots+h_{d-i} \leq h_1+h_2+\cdots+h_{i+1} \;\text{ for }\; 0 \leq i \leq \lfr (d+1)/2 \rfr, 
\end{align}
then we have $\delta_j \leq \delta_{d-j}$ for $1 \leq j \leq \lfr (d-1)/2 \rfr$ and 
$\delta_{d+1-k} < \delta_k$ for $1 \leq k \leq \lfr d/2 \rfr$ when $n \geq \max\{s,d+1-s\}$. 
Moreover, if $h_s > 1$, then we also have $\delta_j < \delta_{d-j}$ for $0 \leq j \leq \lfr (d-1)/2 \rfr$, 
and thus, $(\delta_0,\delta_1,\ldots,\delta_d)$ is strictly alternatingly increasing. 
\end{itemize}
\end{Theorem}

\begin{Remark}\label{ippanka}{\em 
Since the inequalities \eqref{st} and \eqref{hi} are nothing but the inequalites \eqref{Stanley} and \eqref{Hibi} 
for $\delta$-vectors, the $\delta$-vectors of lattice polytopes satisfy all the conditions 
in Theorem \ref{main1} for $h_0,h_1,\ldots,h_d$. 
Therefore, as an immediate corollary of Theorem \ref{main1}, we obtain Theorem \ref{meinseoremu}. 
}\end{Remark}

For the proof of Theorem \ref{main1}, we recall a useful lemma for the computation of $U_n h(t)$. 
\begin{Lemma}[cf. {\cite[Lemma 3.2]{BeckStapledon}}]\label{Veronese}
Let $h(t)=\sum_{j=0}^dh_jt^j$ be the same as in Theorem \ref{main1} and 
let $U_n h(t)=\sum_{j=0}^d \delta_jt^j$. Then one has $\delta_i=c_{n i}$ for each $0 \leq i \leq d$, 
where $c_j$ is the coefficient of the polynomial $h(t) (1+t+\cdots+t^{n-1})^{d+1}$ in $t$. 
\end{Lemma}

Moreover, we also recall the following fundamental assertion for log-concave sequences. 
\begin{Lemma}\label{kihon}
Let $b_1,b_2,\ldots$ be a (resp. strictly) log-concave sequence of nonnegative real numbers such that 
$b_k=b_{k+1}=\cdots=0$ if $b_k=0$ for some $k$. 
Then we have $b_ib_j \geq b_{i-m}b_{j+m}$ (resp. $b_ib_j > b_{i-m}b_{i+m}$) for any $i \leq j$ and $m \geq 0$. 
\end{Lemma}
\begin{proof}
It suffices to show that $b_ib_j \geq b_{i-1}b_{j+1}$ for every $i \leq j$. 
When $b_i=0$ or $b_j=0$, one has $b_{j+1}=0$. Thus the equality holds. 

Assume that $b_i > 0$ and $b_j > 0$. Thus, in particular, 
we have $b_k > 0$ for each $i \leq k \leq j$. 
Since $b_i^2 \geq b_{i-1}b_{i+1}$ and $b_i>0$, one has $b_i \geq b_{i-1}b_{i+1}/b_i$. 
Similarly, since $b_{i+1} \geq b_ib_{i+2}/b_{i+1}$, we obtain 
$b_i \geq b_{i-1}/b_i \cdot b_ib_{i+2}/b_{i+1}=b_{i-1}b_{i+2}/b_{i+1}.$ By repeating this computation, we obtain 
$$b_i \geq b_{i-1}b_{i+1}/b_i \geq b_{i-1}b_{i+2}/b_{i+1} \geq \cdots \geq b_{i-1}b_{j+1}/b_j.$$ 
Hence, $b_ib_j \geq b_{i-1}b_{j+1}$ holds. The case of strictly log-concave sequences is similar. 
\end{proof}

For a sequence of numbers $b_0,b_1,\ldots,$ let $I(b_n)$ denote the index of $b_n$, i.e., $I(b_n)=n$.

\begin{proof}[Proof of Theorem \ref{main1}]
Our goal is to show that $\delta_0,\delta_1,\ldots,\delta_d$ satisfy the following inequalities: 
\begin{enumerate}
\item[(i)] $\delta_i^2 > \delta_{i-1}\delta_{i+1}$ for $1 \leq i \leq d-1$; 
\item[(ii-a)] $\delta_i \leq \delta_{d-i}$ for $1 \leq i \leq \lfr (d-1)/2 \rfr$ and 
$\delta_i < \delta_{d-i}$ for $0 \leq i \leq \lfr (d-1)/2 \rfr$ if $h_s > 1$; 
\item[(ii-b)] $\delta_{d+1-i} < \delta_i$ for $1 \leq i \leq \lfr d/2 \rfr$. 
\end{enumerate}
Let $a_0,a_1,\ldots,a_{(d+1)(n-1)}$ be the integers such that 
$\sum_{i=0}^{(d+1)(n-1)}a_it^i = (1+t+\cdots+t^{n-1})^{d+1}$. By Lemma \ref{Veronese}, we have 
$$\delta_0=1 \text{ and }\delta_i=\sum_{j=0}^s h_j a_{ni-j} \text{ for } 1 \leq i \leq d,$$ 
where we let $a_k=0$ if $k > (d+1)(n-1)$. 

Since the coefficients of the polynomial $(1+t+ \cdots +t^{n-1})^2$ are symmetric and strictly log-concave, 
so are the coefficients of the polynomial $\sum_{i=0}^{(d+1)(n-1)}a_it^i=(1+t+\cdots+t^{n-1})^{d+1}$. 
(See \cite[Proposition 1 and Proposition 2]{StanleyLC}.) In particular, we have 
\begin{align}\label{syma}
&a_i = a_{(d+1)(n-1)-i} \;\text{ for }\; 0 \leq i \leq (d+1)(n-1), \\
\label{lca}
&a_i^2 > a_{i-1}a_{i+1} \; \text{ for } \; 1 \leq i \leq (d+1)(n-1)-1, \text{ and }\\
\label{unia}
&a_i > a_{i-1} \; \text{ for } \; 1 \leq i \leq \lfr (d+1)(n-1)/2 \rfr. 
\end{align}


\noindent
(i) We can compute as follows: 
\begin{align*}
&\delta_i^2-\delta_{i-1}\delta_{i+1} \\
&=\left(\sum_{j=0}^s h_ja_{ni-j}\right)^2 - \left(\sum_{j=0}^s h_ja_{n(i-1)-j}\right)\left(\sum_{j=0}^s h_ja_{n(i+1)-j}\right) \\
&=\sum_{j=0}^s h_j^2(a_{ni-j}^2 - a_{n(i-1)-j}a_{n(i+1)-j}) \\
&+\sum_{0 \leq p < q \leq s}h_ph_q
\left( (a_{ni-p}a_{ni-q} - a_{n(i-1)-p}a_{n(i+1)-q}) + (a_{ni-p}a_{ni-q} - a_{n(i-1)-q}a_{n(i+1)-p}) \right). 
\end{align*}
By \eqref{lca} and Lemma \ref{kihon}, we immediately obtain 
$$a_{ni-j}^2-a_{ni-j-n}a_{ni-j+n} > 0 \text{ and }a_{ni-q}a_{ni-p} - a_{ni-q-n}a_{ni-p+n} > 0.$$ 
Moreover, since $n \geq s$, we have $I(a_{ni-q})-I(a_{n(i-1)-p})=I(a_{n(i+1)-q})-I(a_{ni-p})
=n-q+p \geq n-s \geq 0$. 
Thus we also obtain $$a_{ni-q}a_{ni-p} - a_{n(i-1)-p}a_{n(i+1)-q} > 0$$ by \eqref{lca} and Lemma \ref{kihon}. 
From the nonnegativity of $h_i$ together with $h_0=1$, 
we conclude that $\delta_i^2 - \delta_{i-1}\delta_{i+1} \geq h_0^2 (a_{ni}^2-a_{n(i-1)}a_{n(i+1)})> 0$, as required.

\smallskip


\noindent
(ii-a) Fix $1 \leq i \leq \lfr (d-1)/2 \rfr$. In the sequel, we will prove $\delta_{d-i} - \delta_i \geq 0$ 
and the strictly one holds if $h_s > 1$. Note that $\delta_d \geq h_s > 1 = \delta_0$ for $n \geq \max\{s,d+1-s\}$. 

Let $k_0,\ldots,k_\ell$ be the indices such that 
$\sum_{r=0}^d h_rt^r=h_{k_0}+h_{k_1}t^{k_1}+\cdots + h_{k_\ell} t^{k_\ell}$, 
where $0 = k_0 < k_1 < \cdots < k_\ell = s$ and $h_{k_j}>0$ for each $0 \leq j \leq \ell$. 
For $0 \leq j < \ell$, if $h_{k_0}+\cdots+h_{k_j} \leq h_{k_\ell}=h_s$, then we set $m(j)=\ell$; otherwise 
let $m(j)$ be a unique integer with $0 \leq m(j) < \ell$ such that 
$$h_{k_\ell} + \cdots + h_{k_{m(j)+1}}<h_{k_0} + \cdots + h_{k_j} \leq h_{k_\ell} + \cdots + h_{k_{m(j)}}.$$
Clearly, $m(j-1) \geq m(j)$. Moreover, we have 
\begin{align}\label{jouken}
k_j+k_{m(j)} \geq s. 
\end{align}
In fact, $h_{k_\ell} + \cdots + h_{k_\ell - k_j} \geq h_{k_0} + \cdots + h_{k_j}$ should be satisfied by \eqref{st}, 
while $h_{k_\ell} + \cdots + h_{k_{m(j)+1}}=h_{k_\ell} + \cdots + h_{k_{m(j)}+1} < h_{k_0} + \cdots + h_{k_j}$ holds. 
Thus, we have $k_{m(j)}+1 > k_\ell - k_j,$ i.e., $k_j+k_{m(j)} \geq k_\ell=s$.

Let $t=\max\{j : j \leq m(j)\}$. In particular, we have 
$$0=k_0 < k_1 < \cdots < k_t \leq k_{m(t)} \leq k_{m(t-1)} \leq \cdots \leq k_{m(0)}=k_\ell.$$ 
Then we see that $m(t)=t$ or $m(t)=t+1$. In fact, on the contrary, suppose that $m(t) \geq t+2$. Since $m(t+1) < t+1$, we have 
\begin{align*}
h_{k_0}+\cdots+h_{k_t}+h_{k_{t+1}} &> h_{k_\ell}+\cdots+h_{k_{m(t)}}+\cdots+h_{k_{t+2}}+h_{k_{t+1}}+\cdots + h_{k_{m(t+1)+1}} \\
&\geq h_{k_\ell}+\cdots+h_{k_{m(t)}} + h_{k_{t+1}} \geq h_{k_0}+\cdots+h_{k_t} + h_{k_{t+1}}, 
\end{align*}
a contradiction.

For each $0 \leq p,q \leq \ell$, let $A(p,q)=(a_{n(d-i)-k_p}-a_{ni-k_q})+(a_{n(d-i)-k_q}-a_{ni-k_p})$. 
Then we have $A(j,m(j)) \geq 0$ for $0 \leq j \leq \ell$. 
In fact, by \eqref{syma}, we have $a_{n(d-i)-u}=a_{n(i+1)-d-1+u}$. Thus 
\begin{align*}
I(a_{n(i+1)-d-1+k_j})-I(a_{ni-k_{m(j)}})&=I(a_{n(i+1)-d-1+k_{m(j)}})-I(a_{ni-k_j})\\
&=n-d-1+k_j+k_{m(j)} \geq n-(d+1-s) \geq 0 
\end{align*}
by \eqref{jouken}. When $1 \leq i \leq (d-2)/2$, one has 
\begin{align*}
&\lfr (d+1)(n-1)/2 \rfr - \max\{I(a_{n(i+1)-d-1+k_j}),I(a_{n(i+1)-d-1+k_{m(j)}})\} \\
&\geq ((d+1)(n-1)-1)/2 - (nd/2 - d - 1) - \max\{k_j,k_{m(j)}\} \geq (n+d-2s)/2 \geq 0. 
\end{align*}
When $i=(d-1)/2$, if $k_j \geq (d+1)/2$, then one has 
\begin{align*}
\lfr (d+1)(n-1)/2 \rfr - I(a_{n(d+1)/2-k_j}) \geq k_j - (d+1)/2 \geq 0, 
\end{align*}
and if $k_j \leq (d+1)/2$, since $a_{n(d+1)/2-k_j}=a_{(n-2)(d+1)/2+k_j}$ by \eqref{syma}, one also has \begin{align*}
\lfr (d+1)(n-1)/2 \rfr - I(a_{(n-2)(d+1)/2+k_j}) \geq (d+1)/2 - k_j \geq 0. 
\end{align*}
Similarly, one also has 
$$\lfr (d+1)(n-1)/2 \rfr \geq I(a_{n(d+1)/2-k_{m(j)}}) \text{ or }\lfr (d+1)(n-1)/2 \rfr \geq I(a_{(n-2)(d+1)/2+k_{m(j)}}).$$ 
Therefore, by \eqref{unia}, we conclude that $A(j,m(j)) \geq 0$ for each $0 \leq j \leq t$. 
In the same way, we also conclude that $A(j,r) \geq 0$ for each $m(j) \leq r \leq m(j-1)$. 

Moreover, for $1 \leq j \leq m(t)$, we define $f_j$ by setting 
\begin{align*}
f_j=
\begin{cases}
\left(\sum_{r=m(j-1)}^\ell h_{k_r}-\sum_{r=0}^{j-1}h_{k_r} \right) 
A(j,m(j-1)) +\sum_{r=m(j)+1}^{m(j-1)-1}h_{k_r}A(j,r) \\
\quad+\left(\sum_{r=0}^j h_{k_r}-\sum_{r=m(j)+1}^\ell h_{k_r} \right)A(j,m(j)), \text{ if } m(j)<m(j-1), \\
h_{k_j}A(j,m(j)), \text{ if } m(j)=m(j-1) 
\end{cases}
\end{align*}
for $1 \leq j \leq m(t)-1$, and 
\begin{align*}
f_t=
\begin{cases}
\left(\sum_{r=m(t-1)}^\ell h_{k_r}-\sum_{r=0}^{t-1}h_{k_r} \right) A(t,m(t-1)) +\sum_{r=m(t)+1}^{m(t-1)-1}h_{k_r}A(t,r) \\
\quad+\left(\sum_{r=0}^t h_{k_r}-\sum_{r=m(t)+1}^\ell h_{k_r} \right)A(t,t)/2, \text{ if }m(t)<m(t-1), \\
\left(\sum_{r=m(t-1)}^\ell h_{k_r} - \sum_{r=0}^{t-1}h_{k_r}\right)A(t,t)/2, \text{ if }m(t)=m(t-1) 
\end{cases}
\end{align*}
when $m(t)=t$ and 
\begin{align*}
f_{t+1}=\left(\sum_{r=m(t)}^\ell h_{k_r} - \sum_{r=0}^th_{k_r}\right)A(t+1,t+1)/2
\end{align*}
when $m(t)=t+1$. 
By definition of $m(j)$ together with the nonnegativity of each $A(j,r)$ for $m(j) \leq r \leq m(j-1)$, 
we obtain that $f_j \geq 0$ for each $j$.

By using these notation, we can compute as follows: 
\begin{align*}
\delta_{d-i}-\delta_i=\sum_{j=0}^\ell h_{k_j}a_{n(d-i)-k_j} - \sum_{j=0}^\ell h_{k_j}a_{ni-k_j} 
=A(0,\ell)+\sum_{j=1}^{m(t)}f_j \geq 0. 
\end{align*}
Furthermore, if $h_s >1$, then we have 
\begin{align*}
&f_1 \geq \left(\sum_{r=m(0)}^\ell h_{k_r} - h_{k_0} \right)A(1,m(0)) \geq h_s -1 >0 \text{ when }m(1) < m(0)=\ell, \\
&f_1 = h_{k_1} A(1,m(1)) = h_{k_1} A(1,m(0)) \geq h_{k_1} > 0 \text{ when }m(1)=m(0). 
\end{align*}
Therefore, we also obtain that $\delta_{d-i} > \delta_i$ if $h_s>1$. 

\smallskip


\noindent
(ii-b) Fix $1 \leq i \leq \lfr d/2 \rfr$. In the sequel, we will prove $\delta_i - \delta_{d+1-i} > 0$. 
We employ the similar technique to the above (ii-a).

Let $k_0,\ldots,k_\ell$ be the same things as above. 
For $0 < j \leq \ell$, if $h_{k_1} \geq h_{k_\ell} + \cdots + h_{k_j}$, then we set $n(j)=1$; otherwise 
let $n(j)$ be a unique integer with $1 < n(j) \leq \ell$ such that 
$$h_{k_1}+\cdots + h_{k_{n(j)-1}} < h_{k_\ell} + \cdots + h_{k_j} \leq h_{k_1} + \cdots + h_{k_{n(j)}}.$$
Clearly, $n(j-1) \geq n(j)$. Moreover, we have 
\begin{align}\label{jouken2}
k_j+k_{n(j)} \leq d+1. 
\end{align}
In fact, $$h_{k_1} + \cdots + h_{d+1 - k_j}=h_1+\cdots+h_{d+1-k_j} \geq h_d+\cdots+h_{k_j} = h_{k_\ell} + \cdots + h_{k_j}$$ 
should be satisfied by \eqref{hi}, 
while $h_{k_1} + \cdots + h_{k_{n(j)-1}}=h_1 + \cdots + h_{k_{n(j)}-1} < h_{k_\ell} + \cdots + h_{k_j}$ holds. 
Thus, we have $k_{n(j)}-1 < d+1-k_j,$ i.e., $k_j+k_{n(j)} \leq d+1$.

Let $t'=\min\{j : j \geq n(j)\}$. In particular, we have 
$$1 \leq k_{n(\ell)} \leq k_{n(\ell-1)} \leq \cdots \leq k_{n(t')} \leq k_{t'} < k_{t'+1} < \cdots < k_\ell.$$ 
Then $n(t')=t'$ or $n(t')=t'-1$. In fact, on the contrary, suppose that $n(t') \leq t'-2$. Since $n(t'-1) > t'-1$, we have 
\begin{align*}
h_{k_\ell}+\cdots+h_{k_{t'-1}} &> h_{k_1}+\cdots+h_{k_{n(t')}}+\cdots+h_{k_{t'-2}}+h_{k_{t'-1}}+\cdots+h_{k_{n(t'-1)-1}} \\
&\geq h_{k_1}+\cdots+h_{k_{n(t')}}+h_{k_{t'-1}} \geq h_{k_\ell}+\cdots+h_{k_{t'}} + h_{k_{t'-1}}, 
\end{align*}
a contradiction.

For each $0 \leq p,q \leq \ell$, let $B(p,q)=(a_{ni-k_p}-a_{n(d+1-i)-k_q})+(a_{ni-k_q}-a_{n(d+1-i)-k_p})$. 
Then we have $B(j,n(j)) \geq 0$. In fact, by \eqref{syma}, we have $a_{n(d+1-i)-u}=a_{ni-d-1+u}$. Thus 
\begin{align*}
I(a_{ni-k_j})-I(a_{ni-d-1+k_{n(j)}})=I(a_{ni-k_{n(j)}})-I(a_{ni-d-1+k_j})=d+1-(k_j+k_{n(j)}) \geq 0 
\end{align*}
by \eqref{jouken2}. When $1 \leq i \leq (d-1)/2$, one has 
\begin{align*}
&\lfr (d+1)(n-1)/2 \rfr - \max\{I(a_{ni-k_j}),I(a_{ni-k_{n(j)}})\} \\
&\geq ((d+1)(n-1)-1)/2 - n(d-1)/2 + \min\{k_j,k_{n(j)}\} \geq n-d/2 > 0 
\end{align*}
by $n \geq (d+1)/2$. When $i=d/2$, if $n/2 + k_j \geq d/2+1$, then one has \begin{align*}
\lfr (d+1)(n-1)/2 \rfr - I(a_{nd/2-k_j}) \geq (n-d-2)/2 + k_j \geq 0, 
\end{align*}
and if $n/2 + k_j \leq d/2$, since $a_{nd/2-u}=a_{n(d/2+1)-d-1+u}$ by \eqref{syma}, one also has 
\begin{align*}
\lfr (d+1)(n-1)/2 \rfr - I(a_{n(d/2+1)-d-1+k_j}) \geq (d-n)/2 - k_j \geq 0. 
\end{align*}
Similarly, one also has 
$$\lfr (d+1)(n-1)/2 \rfr \geq I(a_{nd/2-k_{n(j)}}) \text{ or }\lfr (d+1)(n-1)/2 \rfr \geq I(a_{n(d/2+1)-d-1+k_{n(j)}}).$$ 
Therefore, by \eqref{unia}, we conclude that $B(j,n(j)) \geq 0$ for each $1 \leq j \leq \ell$. 
In the same way, we also conclude that $B(j,r) \geq 0$ for $n(j+1) \leq r \leq n(j)$. 

Moreover, for $n(t') \leq j \leq \ell$, we define $g_j$ by setting 
\begin{align*}
g_j=
\begin{cases}
\left(\sum_{r=1}^{n(j+1)} h_{k_r}-\sum_{r=j+1}^\ell h_{k_r} \right) 
B(j,n(j+1)) +\sum_{r=n(j+1)+1}^{n(j)-1}h_{k_r}B(j,r) \\
\quad+\left(\sum_{r=j}^\ell h_{k_r}-\sum_{r=1}^{n(j)-1} h_{k_r} \right)B(j,n(j)), \text{ if } n(j+1)<n(j), \\
h_{k_j}B(j,n(j)), \text{ if } n(j+1)=n(j), 
\end{cases}
\end{align*}
for $n(t')+1 \leq j \leq \ell$, where we let $n(\ell+1)=0$, and 
\begin{align*}
g_{t'}=
\begin{cases}
\left(\sum_{r=1}^{n(t'+1)} h_{k_r}-\sum_{r=t'+1}^\ell h_{k_r} \right) B(t',n(t'+1)) 
+\sum_{r=n(t'+1)+1}^{n(t')-1}h_{k_r}B(t',r) \\
\quad+\left(\sum_{r=t'}^\ell h_{k_r}-\sum_{r=1}^{n(t')-1} h_{k_r} \right)B(t',t')/2, \text{ if }n(t'+1)<n(t'), \\
\left(\sum_{r=1}^{n(t')} h_r - \sum_{r=t'+1}^\ell h_r\right)B(t',t')/2, \text{ if }n(t'+1)=n(t') 
\end{cases}
\end{align*}
when $n(t')=t'$ and 
\begin{align*}
g_{t'-1}=\left(\sum_{r=1}^{n(t')} h_r - \sum_{r=t'}^\ell h_r\right)B(t'-1,t'-1)/2
\end{align*}
when $n(t')=t'-1$. 
By definition of $n(j)$ together with the nonnegativity of each $B(j,r)$ for $n(j+1) \leq r \leq n(j)$, 
we obtain that $g_j \geq 0$ for each $j$.

By using these notation, we can compute as follows: 
\begin{align*}
\delta_i-\delta_{d+1-i}=\sum_{j=0}^\ell h_{k_j}a_{ni-k_j} - \sum_{j=0}^\ell h_{k_j}a_{n(d+1-i)-k_j} 
=B(0,0)/2+\sum_{j=n(t')}^\ell g_j. 
\end{align*}
Since $B(0,0)>0$ and $g_j \geq 0$ for each $j$, we have $\delta_i-\delta_{d+1-i} > 0$, as required. 
\end{proof}

\section{Several examples of $\delta$-vectors concerning unimodality}

The goal of this section is to provide several kinds of $\delta$-vectors. 
Those concern unimodality, log-concavity and alternatingly increasingness.

\begin{Remark}\label{chuui}{\em 
(a) Let $\Pc \subset \RR^N$ be a lattice polytope 
and $(\delta_0,\delta_1,\ldots,\delta_d)$ its $\delta$-vector. 
If $\Pc$ has IDP, then one has $\delta_1^2 \geq \delta_0\delta_2$.

In fact, let $\delta_1=\ell$. Then $|\Pc \cap \ZZ^N|=\ell+d+1$. 
If $\ell=0$, then we do not have to say anything from \cite[Lemma 3.1]{SchLan}. 
Assume $\ell > 0$. From $i(\Pc,m)=\sum_{i=0}^d \delta_i\binom{m+d-i}{d}$, 
we also see that $|2\Pc \cap \ZZ^N|=\binom{d+2}{2}+(d+1)\ell+\delta_2$. 
Since $\Pc$ has IDP, we have $|m\Pc \cap \ZZ^N| \leq \binom{\ell+d+m}{m}$. 
In particular, $|2\Pc \cap \ZZ^N| \leq \binom{\ell+d+2}{2}$. 
Hence $\delta_2 \leq \binom{\ell+d+2}{2}-\binom{d+2}{2}-(d+1)\ell=(\ell^2+\ell)/2$. Therefore, 
$\delta_1^2-\delta_0\delta_2=\ell^2-\delta_2 \geq \ell^2-(\ell^2+\ell)/2=\ell(\ell-1)/2$. 
This is always nonnegative by $\ell >0$, as required.

Moreover, one has $\delta_2 \geq \delta_1$ by \cite{HibiLBT}. 
Note that $\delta_1 \geq \delta_d$ always holds. (See Introduction.) Thus, we have $\delta_2^2 \geq \delta_1 \delta_d$. 
Hence, we obtain that all $\delta$-vectors of lattice polytopes of dimension at most 3 having IDP are always log-concave. 

\noindent
(b) The $\delta$-vectors of lattice polytopes of dimension at most 4 with at least one interior lattice point 
are always alternatingly increasing. In particular, it is unimodal. 
See \cite[Proposition 3.2]{SchLan}. 
}\end{Remark}

Before providing examples, 
we recall the well-known combinatorial technique how to compute 
the $\delta$-vector of a lattice simplex. 
Given a lattice simplex $\Fc \subset \RR ^N$ of dimension $d$ with the vertices 
$v_0, v_1, \ldots, v_d \in \ZZ^N$, we set 
\begin{eqnarray*}
\Lambda_\Fc=\left\{ \alpha \in \ZZ^{N+1} : \alpha = \sum_{i=0}^d r_i(v_i,1), \; 0 \leq r_i < 1 \right\}. 
\end{eqnarray*}
We define the degree of $\alpha=\sum_{i=0}^{d}r_i(v_i,1) \in \Lambda_\Fc$ to be $\deg(\alpha)=\sum_{i=0}^d r_i$, 
i.e., the last coordinate of $\alpha$. Then we have the following: 
\begin{Lemma}[cf. {\cite[Corollary 3.11]{BeckRobins}}]\label{compute}
Let $\delta(\Fc)=(\delta_0,\delta_1,\ldots,\delta_d).$ Then, for each $0 \leq i \leq d$, 
$$\delta_i = |\{ \alpha \in \Lambda_\Fc : \deg(\alpha)=i\}|.$$ 
\end{Lemma}

\subsection{Non-unimodal $\delta$-vectors}

First, we construct lattice polytopes which contain $m$ interior lattice points 
whose $\delta$-vectors are not unimodal. 
By Remark \ref{chuui} (b), if a lattice polytope has a non-unimodal $\delta$-vector, 
then its dimension is at least 5. 

Let $\eb_1,\ldots,\eb_d$ be the unit coordinate vectors of $\RR^d$ and ${\bf 0}$ its origin.

\begin{Example}\label{main2}{\em 
Let $d \geq 5$ and $m \geq 1$ be integers. 
Then there exists a lattice polytope of dimension $d$ containing exactly $m$ interior lattice points 
such that its $\delta$-vector is not unimodal.

\noindent
\underline{The case $d$ is odd}: Let $d=2 \ell +1$, where $\ell \geq 2$. 
We define $\Pc_\text{odd}(\ell,m)$ by setting the convex hull of $v_0,v_1,\ldots,v_d$, where $M=2(2m+1)(\ell+1)$ and 
\begin{eqnarray*}
&&v_i=
\begin{cases}
{\bf 0}, \quad &i=0, \\
\eb_i, \quad &i=1,\ldots,d-1,
\end{cases}\\
&&v_d=(M-2(\ell+1)m)\eb_1+(M-1)(\eb_2+\cdots+\eb_{d-1})+M\eb_d. 
\end{eqnarray*}
Let $\delta(\Pc_\text{odd}(\ell,m))=(\delta_0,\delta_1,\ldots,\delta_d)$. 
Then we can calculate from Lemma \ref{compute} that 
$$\delta_i=\left| \left\{ j \in \ZZ : 
\lcl \frac{2 \ell j}{M}+\left\{ \frac{2(\ell+1)mj}{M}\right\} \rcl = i, \; 0 \leq j \leq M-1 \right\} \right|,$$ 
where $\{r\}$ denotes the fraction part of a rational number $r$, i.e., $\{r\}=r - \lfr r \rfr$. 
Let $f(j)=\lcl \frac{2 \ell j}{M}+\left\{ \frac{2(\ell+1)mj}{M}\right\} \rcl
=\lcl \frac{ \ell j}{(2m+1)(\ell+1)}+\left\{ \frac{mj}{2m+1}\right\} \rcl$. 
\begin{itemize}
\item Let $j=(2m+1)k+2p+r$, where $0 \leq k \leq 2\ell+1$, $0 \leq p \leq m-1$ and $r=1,2$. Then 
$$f(j)=\lcl \frac{\ell k}{\ell+1}+\frac{rm}{2m+1}+\frac{(\ell-1)p+\ell r}{(2m+1)(\ell+1)}\rcl.$$
\item Let $j=(2m+1)k$, where $1 \leq k \leq 2 \ell+1$. Then 
$$f(j)=\lcl \frac{\ell k}{\ell+1}\rcl.$$ 
\end{itemize}

(i) We prove that $\Pc_\text{odd}(\ell,m)$ contains exactly $m$ lattice points in its interior, 
i.e., we may check $\delta_d=\delta_{2\ell+1}=m$. 
\begin{itemize}
\item[(a)] For $j=(2m+1)k+2p+r$, where $0 \leq k \leq 2\ell+1$, $0 \leq p \leq m-1$ and $r=1,2$, 
if $k \leq 2 \ell$, then we see that 
\begin{align*}
f(j) &= \lcl \frac{\ell k}{\ell+1}+\frac{rm}{2m+1}+\frac{p(\ell-1)+\ell r}{(2m+1)(\ell+1)}\rcl \\
&\leq \lcl \frac{2\ell^2}{\ell+1}+\frac{2m}{2m+1}+\frac{(m-1)(\ell-1)+2\ell}{(2m+1)(\ell+1)}\rcl \\
&=2 \ell -1 + \lcl \frac{m\ell+3m+2}{(2m+1)(\ell+1)}\rcl \leq 2 \ell. 
\end{align*}
Thus, if $f(j)=2\ell+1$, then $k=2\ell+1$. Similarly, if $r=1$, then we see that $f(j) \leq 2 \ell$. 
Thus, if $f(j)=2\ell+1$, then $r=2$. On the other hand, when $k=2\ell+1$ and $r=2$, 
we obtain that 
$$f(j)=2\ell+\lcl \frac{p(\ell-1)+\ell + 2m}{(2m+1)(\ell+1)} \rcl=2\ell+1 \; 
\;\text{ for each }\; 0 \leq p \leq m-1.$$
\item[(b)]
For $j=(2m+1)k$, where $1 \leq k \leq 2 \ell+1$, we see that $f(j) \leq 2\ell$. 
\end{itemize}
From the above (a) and (b), we conclude that $\delta_{2 \ell+1}=\delta_d=m$. 

(ii) We prove the non-unimodality of $(\delta_0,\ldots,\delta_d)$. 
\begin{itemize}
\item[(a)] The following statements imply that $\delta_1 \leq m+1$. 
\begin{itemize}
\item For $j=(2m+1)k+2p+r$, where $0 \leq k \leq 2\ell+1$, $0 \leq p \leq m-1$ and $r=1,2$, 
if $k \geq 1$, then we have $f(j) \geq 2$. Similarly, if $r=2$, then $f(j) \geq 2$. 
Thus, if $f(j)=1$, then $k=0$ and $r=1$. 
\item Moreover, for $j=(2m+1)k$, where $1 \leq k \leq 2 \ell+1$, we see that $f(j) =1$ 
if and only if $j=1$. 
\end{itemize}
\item[(b)] The following imply that $\delta_\ell \geq 2m+2$. 
\begin{itemize}
\item For $j=(2m+1)k+2p+r$, where $0 \leq k \leq 2\ell+1$, $0 \leq p \leq m-1$ and $r=1,2$, 
if $k=\ell$ and $r=1$, then we see that $f(j)=\ell$. 
Similarly, if $k=\ell-1$ and $r=2$, then we see that $f(j)=\ell$. 
\item Moreover, for $j=(2m+1)k$, where $1 \leq k \leq 2 \ell+1$, one has $f(j) =\ell$ 
if and only if $k=\ell$ or $k=\ell+1$. 
\end{itemize}
\item[(c)] The following imply that $\delta_{\ell+1} \leq 2m+1$. 
\begin{itemize}
\item For $j=(2m+1)k+2p+r$, where $0 \leq k \leq 2\ell+1$, $0 \leq p \leq m-1$ and $r=1,2$, 
if $k \leq \ell-1$, then $f(j) \leq \ell$. Moreover, $k \geq \ell+2$, then $f(j) \geq \ell+2$. 
In addition, if $k=\ell+1$ and $r=2$, then $f(j) \geq \ell+2$. 
Furthermore, if $k=\ell$ and $r=1$, then $f(j) \leq \ell$. 
Thus, it must be satisfied that 
$(k,r)=(\ell,2)$ or $(k,r)=(\ell+1,1)$ when $f(j)=\ell+1$. 
\item Moreover, for $j=(2m+1)k$, where $1 \leq k \leq 2 \ell+1$, one has $f(j) =\ell+1$ 
if and only if $k=\ell+2$. 
\end{itemize}
\item[(d)] The following imply that $\sum_{i=\ell+2}^{2\ell}\delta_i \geq 2m\ell+\ell-1$. 
Then we notice that $2m\ell+\ell-1=(\ell-1)(2m+1)+2m$. 
Hence, we obtain that $\max\{\delta_{\ell+2},\ldots,\delta_{2\ell}\} \geq 2m+2$. 
\begin{itemize}
\item For $j=(2m+1)k+2p+r$, where $0 \leq k \leq 2\ell+1$, $0 \leq p \leq m-1$ and $r=1,2$, 
if $\ell +2 \leq k \leq 2\ell$, then $\ell + 2 \leq f(j) \leq 2\ell$. 
Moreover, if $k = 2\ell+1$ and $r=1$, then $f(j) =2\ell$. 
In addition, if $k=\ell+1$ and $r=2$, then $f(j)=\ell+2$. 
\item Moreover, for $j=(2m+1)k$, where $1 \leq k \leq 2 \ell+1$, one has $\ell+2 \leq f(j) \leq 2\ell$ 
if and only if $\ell+3 \leq k \leq 2\ell+1$. 
\end{itemize}
\end{itemize}
Summarizing the above (a)--(d), one sees that 
$$\delta_1 \leq m+1, \; \delta_\ell \geq 2m+2, \; \delta_{\ell+1} \leq 2m+1 \text{ and }
\max\{\delta_{\ell+2},\ldots,\delta_{2\ell}\} \geq 2m+2.$$ 
Hence, 
$$\delta_1 < \delta_\ell > \delta_{\ell+1} < \max\{\delta_{\ell+2},\ldots,\delta_{2\ell}\}.$$ 
This shows the non-unimodality of $(\delta_0,\delta_1,\ldots,\delta_{2\ell+1})$. 

\bigskip

\noindent
\underline{The case $d$ be even}: Let $d=2 \ell +2$, where $\ell \geq 2$. 
We define $\Pc_\text{even}(\ell,m)$ by setting the convex hull of $v_0,v_1,\ldots,v_d$, where $M=2(3m+1)(\ell+1)$ and 
\begin{eqnarray*}
&&v_i=
\begin{cases}
{\bf 0}, \quad &i=0, \\
\eb_i, \quad &i=1,\ldots,d-1,
\end{cases}\\
&&v_d=(M-2(\ell+1)m)(\eb_1+\eb_2)+(M-1)(\eb_3+\cdots+\eb_{d-1})+M\eb_d. 
\end{eqnarray*}
Let $\delta(\Pc_\text{even}(\ell,m))=(\delta_0,\delta_1,\ldots,\delta_d)$. 
Then we can calculate from Lemma \ref{compute} that 
$$\delta_i=\left| \left\{ j \in \ZZ : 
\lcl \frac{2 \ell j}{M}+2\left\{ \frac{2(\ell+1)mj}{M}\right\} \rcl = i, \; 0 \leq j \leq M-1 \right\} \right|.$$ 
Let $g(j)=\lcl \frac{2 \ell j}{M}+2\left\{ \frac{2(\ell+1)mj}{M}\right\} \rcl
=\lcl \frac{\ell j}{(3m+1)(\ell+1)}+2\left\{ \frac{mj}{3m+1}\right\} \rcl$. 
\begin{itemize}
\item Let $j=(3m+1)k+3p+r$, where $0 \leq k \leq 2\ell+1$, $0 \leq p \leq m-1$ and $r=1,2,3$. Then 
$$g(j)=\lcl \frac{\ell k}{\ell+1}+\frac{2rm}{3m+1}+\frac{(\ell-2)p+\ell r}{(3m+1)(\ell+1)}\rcl.$$
\item Let $j=(3m+1)k$, where $1 \leq k \leq 2 \ell+1$. Then 
$$g(j)=\lcl \frac{\ell k}{\ell+1}\rcl.$$ 
\end{itemize}

(i) We prove that $\Pc_\text{even}(\ell,m)$ contains exactly $m$ lattice points in its interior, 
i.e., we may check $\delta_d=\delta_{2\ell+2}=m$. 
\begin{itemize}
\item[(a)] For $j=(2m+1)k+3p+r$, where $0 \leq k \leq 2\ell+1$, $0 \leq p \leq m-1$ and $r=1,2,3$, 
if $k \leq 2 \ell$, then we see that $g(j) \leq 2\ell+1$. 
Thus, if $g(j)=2\ell+2$, then $k=2\ell+1$. Similarly, if $r \leq 2$, then we see that $f(j) \leq 2 \ell+1$. 
Thus, if $g(j)=2\ell+2$, then $r=3$. On the other hand, when $k=2\ell+1$ and $r=3$, 
we obtain that 
$$g(j)=2\ell+1+\lcl \frac{p(\ell-2)+\ell + 3m-1}{(3m+1)(\ell+1)} \rcl = 2\ell+2 
\;\text{ for }\; 0 \leq p \leq m-1.$$
\item[(b)]
For $j=(3m+1)k$, where $1 \leq k \leq 2 \ell+1$, we see that $g(j) \leq 2\ell+1$. 
\end{itemize}
From the above (a) and (b), we conclude that $\delta_{2 \ell+2}=\delta_d=m$. 

(ii) We prove the non-unimodality of $(\delta_0,\ldots,\delta_d)$. 
\begin{itemize}
\item[(a)] The following statements imply that $\delta_1 \leq m+1$. 
\begin{itemize}
\item For $j=(3m+1)k+3p+r$, where $0 \leq k \leq 2\ell+1$, $0 \leq p \leq m-1$ and $r=1,2,3$, 
if $k \geq 1$, then we have $g(j) \geq 2$. Similarly, if $r \geq 2$, then $g(j) \geq 2$. 
Thus, if $f(j)=1$, then $k=0$ and $r=1$. 
\item Moreover, for $j=(3m+1)k$, where $1 \leq k \leq 2 \ell+1$, we see that $g(j) =1$ 
if and only if $j=1$. 
\end{itemize}
\item[(b)] The following imply that $\delta_\ell \geq 3m+2$. 
\begin{itemize}
\item For $j=(3m+1)k+3p+r$, where $0 \leq k \leq 2\ell+1$, $0 \leq p \leq m-1$ and $r=1,2,3$, 
if $(k,r)=(\ell-2,3),(\ell-1,2)$ or $(\ell,1)$, then we see that $g(j)=\ell$. 
\item Moreover, for $j=(3m+1)k$, where $1 \leq k \leq 2 \ell+1$, one has $g(j) =\ell$ 
if and only if $k=\ell$ or $k=\ell+1$. 
\end{itemize}
\item[(c)] The following imply that $\delta_{\ell+1} \leq 3m+1$. 
\begin{itemize}
\item For $j=(3m+1)k+3p+r$, where $0 \leq k \leq 2\ell+1$, $0 \leq p \leq m-1$ and $r=1,2,3$, 
we see that $g(j) = \ell+1$ only if $(k,r)=(\ell-1,3),(\ell,2)$ or $(\ell+1,1)$. 
\item Moreover, for $j=(3m+1)k$, where $1 \leq k \leq 2 \ell+1$, one has $g(j) =\ell+1$ 
if and only if $k=\ell+2$. 
\end{itemize}
\item[(d)] The following imply that $\sum_{i=\ell+2}^{2\ell}\delta_i \geq 3m\ell+\ell-1$. 
Then we notice that $3m\ell+\ell-1=(\ell-1)(3m+1)+3m$. Hence, we obtain that $\max\{\delta_{\ell+2},\ldots,\delta_{2\ell}\} \geq 3m+2$. 
\begin{itemize}
\item For $j=(3m+1)k+3p+r$, where $0 \leq k \leq 2\ell+1$, $0 \leq p \leq m-1$ and $r=1,2,3$, 
if $\ell+2 \leq k \leq 2\ell-1$, then $\ell+2 \leq g(j) \leq 2 \ell$. Moreover, 
if $(k,r)=(\ell,3),(\ell+1,2),(\ell+1,3),(2\ell,1),(2\ell,2),(2\ell+1,1)$, then $\ell+2 \leq g(j) \leq 2\ell$. 
\item Moreover, for $j=(3m+1)k$, where $1 \leq k \leq 2 \ell+1$, one has $\ell+2 \leq g(j) \leq 2\ell$ 
if and only if $\ell+3 \leq k \leq 2\ell+1$. 
\end{itemize}
\end{itemize}
Summarizing the above (a)--(d), one sees that 
$$\delta_1 \leq m+1, \; \delta_\ell \geq 3m+2, \; \delta_{\ell+1} \leq 3m+1 \text{ and } 
\max\{\delta_{\ell+2},\ldots,\delta_{2\ell}\} \geq 3m+2.$$
Hence, 
$$\delta_1 < \delta_\ell > \delta_{\ell+1} < \max\{\delta_{\ell+2},\ldots,\delta_{2\ell}\},$$ 
as desired. 
}\end{Example}

\subsection{Unimodal but neither log-concave nor alternatingly increasing $\delta$-vectors}

Next, we give examples of lattice polytopes whose $\delta$-vectors are unimodal 
but neither log-concave nor alternatingly increasing for odd dimensions.

\begin{Example}\label{ex1}{\em 
Let $d \geq 5$ be an odd number and $m \geq 1$ an integer. 
We define $\Pc(d,m)$ by setting the convex hull of $v_0,v_1,\ldots,v_d$, where $M=2(d-1)m+2$ and 
\begin{eqnarray*}
v_i=
\begin{cases}
{\bf 0}, &i=0, \\
\eb_i, &i=1,\ldots,d-1, \\
(M-d+1)\eb_1+(M-1)(\eb_2+\cdots+\eb_{d-1})+M\eb_d, &i=d. 
\end{cases}
\end{eqnarray*}

Then it can be computed that $\delta(\Pc(d,m))=(\delta_0,\delta_1,\ldots,\delta_d)$ is equal to 
\begin{eqnarray*}
\delta_i=\left|\left\{ j \in \ZZ : \left\lceil \frac{(d-1)j}{2(d-1)m+2}+ 
\left\{ \frac{(d-1)j}{2(d-1)m+2}\right\}\right\rceil = i, \; 0 \leq j \leq M-1 \right\}\right|. 
\end{eqnarray*}
Let $f(j)=\left\lceil \frac{(d-1)j}{2(d-1)m+2}+\left\{ \frac{(d-1)j}{2(d-1)m+2}\right\}\right\rceil$. 
For $1 \leq j \leq (d-1)m+1$, we see the following: 
\begin{itemize}
\item One has $f(j)=1$ if $1 \leq j \leq m$; 
\item For $2 \leq k \leq (d-1)/2$, we have $f(j)=k$ if $(2k-3)m+1 \leq j \leq (2k-1)m$; 
\item One has $f(j)=(d+1)/2$ if $(d-2)m+1 \leq j \leq (d-1)m$; 
\item One has $f((d-1)m+1)=(d-1)/2$. 
\end{itemize}
For $(d-1)m+2 \leq j \leq 2(d-1)m+1$, 
it is easy that $f(j)=(d-1)/2+f(j-(d-1)m-1)$. 
Therefore, we conclude that 
$$\delta(\Pc(d,m))=(1,m,2m,\ldots,2m,\underbrace{2m+1}_{\delta_{(d-1)/2}},2m,\ldots,2m,m).$$
Since $\delta_{(d-1)/2}>\delta_{(d+1)/2}$ and $\delta_{(d-1)/2}\delta_{(d+3)/2}>\delta_{(d+1)/2}^2$, 
this $\delta$-vector is neither log-concave nor alternatingly increasing. 
On the other hand, this $\delta$-vector is unimodal. 
}\end{Example}

\subsection{Alternatingly increasing but not log-concave $\delta$-vectors}

Next, we give examples of lattice polytopes whose $\delta$-vectors are 
alternatingly increasing but not log-concave. 


\begin{Example}\label{ex2}{\em 
Let $d \geq 4$ and $m \geq 1$ be integers. 
We define $\Pc(d,m)$ by setting the convex hull of $v_0,v_1,\ldots,v_d$, 
where $M=(\lcl (d+1)/2 \rcl m+1)\lcl (d+2)/2 \rcl$ and 
\begin{align*}
&v_i=
\begin{cases}
{\bf 0}, &i=0, \\
\eb_i, &i=1,\ldots,d-1, 
\end{cases} \\
&v_d=(M-\lcl (d+2)/2 \rcl m)(\eb_1+\cdots + \eb_{\lfr d/2 \rfr})+
(M-1)(\eb_{\lfr d/2 \rfr + 1}+\cdots+\eb_{d-1})+M\eb_d. 
\end{align*}


\noindent
\underline{The case $d$ is odd}: 
Let $d'=(d+1)/2$. Then $M=(d'm+1)(d'+1)$. 
It can be computed that $\delta(\Pc(d,m))=(\delta_0,\delta_1,\ldots,\delta_d)$ 
is equal to 
$$\delta_i=\left| \left\{ j \in \ZZ : \lcl \frac{d'j}{(d'm+1)(d'+1)} +
\left\{ \frac{mj}{d'm+1} \right\}(d'-1) \rcl = i, \; 0 \leq j \leq M-1 \right\} \right|.$$
Let $f(j)=\lcl \frac{d'j}{(d'm+1)(d'+1)} +
\left\{ \frac{mj}{d'm+1} \right\}(d'-1) \rcl$.

For each $j=1,\ldots,M-1$, we have a unique expression such that 
$\ell(d'm+1)$ for some $1 \leq \ell \leq d'$ or $j=p(d'm+1)+qd'+r$, 
where $0 \leq p \leq d'$, $0 \leq q \leq m-1$ and $1 \leq r \leq d'$. Thus 
\begin{align*}
f(j)&=\ell-1+\lcl \frac{d' +1-\ell}{d'+1} \rcl=\ell \text{ if }j=\ell (d'm+1), \text{ and }\\
f(j)&= p+r-1 + \lcl \frac{d'+1-p-r}{d'+1}-\frac{mr-q}{(d'm+1)(d'+1)} \rcl
\text{ if }j=p(d'm+1)+qd'+r. 
\end{align*}
Note that $1 \leq mr-q \leq d'm$. 
Hence we obtain that 
\begin{align*}
&f(j)=i \; \text{ for }\; 1 \leq i \leq d'-1 \;\; \Longleftrightarrow \;\; \ell=i\text{ or }p+r=i, \\
&f(j)=d' \;\; \Longleftrightarrow \;\; \ell=d' \text{ or } p+r=d' \text{ or }p+r=d'+1, \\
&f(j)=i \; \text{ for }\; d'+1 \leq i \leq d \;\; \Longleftrightarrow \;\;p+r=i+1. 
\end{align*}
From these observations, we conclude that 
\begin{align*}
\delta(\Pc(d,m))=(1,m+1,2m+1,\ldots,(d'-1)m+1,
\underbrace{2d'm+1}_{\delta_{d'}},(d'-1)m,\ldots,m). 
\end{align*}
Clearly, this is alternatingly increasing, while this is not log-concave by 
$\delta_{d'}\delta_{d'+2}>\delta_{d'+1}^2$. 

\bigskip

\noindent
\underline{The case $d$ is even}: 
Let $d'=d/2$. Then $M=(d'm+m+1)(d'+1)$. 
It can be computed that $\delta(\Pc(d,m))=(\delta_0,\delta_1,\ldots,\delta_d)$ is equal to 
$$\delta_i=\left| \left\{ j \in \ZZ : \lcl \frac{d'j}{(d'm+m+1)(d'+1)} +
\left\{ \frac{mj}{d'm+m+1} \right\} d' \rcl = i, \; 0 \leq j \leq M-1 \right\} \right|.$$
Let $f(j)=\lcl \frac{d'j}{(d'm+m+1)(d'+1)} +
\left\{ \frac{mj}{d'm+m+1} \right\}d'\rcl$.

For each $1 \leq j \leq M-1$, we have a unique expression such that 
$\ell(d'm+m+1)$ for some $1 \leq \ell \leq d'$ or $j=p(d'm+m+1)+q(d'+1)+r$, 
where $0 \leq p \leq d'$, $0 \leq q \leq m-1$ and $1 \leq r \leq d'+1$. Thus 
\begin{align*}
&f(j)=
\ell-1+\lcl \frac{d'+1- \ell}{d'+1} \rcl =\ell \text{ if }j=\ell (d'm+m+1), \\
&f(j)=
p+r-1+\lcl \frac{d'+1-p-r}{d'+1}\rcl \text{ if }j=p(d'm+m+1)+q(d'+1)+r.
\end{align*}
Hence we obtain that 
\begin{align*}
&f(j)=i \; \text{ for }\; 1 \leq i \leq d'-1 \;\; \Longleftrightarrow \;\; \ell=i\text{ or }p+r=i, \\
&f(j)=d' \;\; \Longleftrightarrow \;\; \ell=d' \text{ or } p+r=d' \text{ or } p+r=d'+1, \\
&f(j)=i \; \text{ for }\; d'+1 \leq i \leq d \;\; \Longleftrightarrow \;\; p+r=i+1. 
\end{align*}
From these observations, we conclude that 
\begin{align*}
\delta(\Pc(d,m))=(1,m+1,2m+1,\ldots,(d'-1)m+1,
\underbrace{(2d'+1)m+1}_{\delta_{d'}},d'm,(d'-1)m,\ldots,m). 
\end{align*}
Clearly, this is alternatingly increasing, while this is not log-concave by $\delta_{d'}\delta_{d'+2}>\delta_{d'+1}^2.$ 
}\end{Example}

\smallskip

For the case of lattice polytopes of dimension 3, we note the following: 
\begin{Remark}{\em 
Let $(\delta_0,\delta_1,\delta_2,\delta_3)$ be the $\delta$-vector of 
some lattice polytope of dimension 3 with $\delta_3 \not=0$. 
Since $\delta_2 \geq \delta_1 \geq \delta_3$, we always have $\delta_2^2 \geq \delta_1\delta_3$. 
Moreover, as mentioned in Remark \ref{chuui} (b), $(\delta_0,\delta_1,\delta_2,\delta_3)$ is 
always alternatingly increasing. Thus, if $(\delta_0,\delta_1,\delta_2,\delta_3)$ is alternatingly increasing 
but not log-concave, then it should be $\delta_2 > \delta_1^2$. 
On the other hand, such a lattice polytope never has IDP by Remark \ref{chuui} (a). 
}\end{Remark}

For example, the $\delta$-vector of the lattice polytope with its vertices 
$\eb_1,\eb_2,\eb_3,2(\eb_1+\eb_2+\eb_3)$ is equal to $(1,1,2,1)$. This is not log-concave.

\subsection{Log-concave but not alternatingly increasing $\delta$-vectors}

Finally, we supply a cupple of examples of lattice polytopes 
whose $\delta$-vectors are log-concave but not alternatingly increasing in law dimensions. 

\bigskip

Let $\Pc_3 \subset \RR^6$ be a lattice polytope of dimension 6 whose vertices are 
$${\bf 0},\eb_1,\ldots,\eb_4,2(\eb_1+\cdots+\eb_4)+3\eb_5,16(\eb_1+\cdots+\eb_4)+3\eb_5+30\eb_6.$$
Then we have $\delta(\Pc_3)=(1,6,20,22,23,15,3)$. Moreover, let $\Pc_4$ be a lattice polytope whose vertices are 
$${\bf 0},\eb_1,\ldots,\eb_4,2(\eb_1+\cdots+\eb_4)+3\eb_5,22(\eb_1+\cdots+\eb_4)+3\eb_5+42\eb_6.$$
Then we have $\delta(\Pc_4)=(1,7,28,31,32,23,4)$. Both of them are log-concave but not alternatingly increasing. 

Similarly, we have checked the existence of some more lattice polytopes of dimension 6 
whose $\delta$-vectors are log-concave but not alternatingly increasing.

\subsection{Future works}

We remain the following problems: 
\begin{Problem} If there exists, 
construct a family of lattice polytopes whose $\delta$-vectors are \\
{\em (a)} unimodal but neither log-concave nor alternatingly increasing for even dimensions; \\
{\em (b)} alternatingly increasing but not log-concave for dimension 3; \\
{\em (c)} log-concave but not alternatingly increasing for dimension at least 5. 
\end{Problem}

\bigskip


%
%
%

\end{document}